\documentclass[11pt]{article}
\usepackage[centertags]{amsmath}
\usepackage{amsfonts}
\usepackage{amssymb}
\usepackage{amsthm}
\usepackage{newlfont}
\usepackage[pdftex]{graphicx}
\usepackage{epstopdf}

\newtheorem{Theorem}{Theorem}[section]
\newtheorem{Definition}[Theorem]{Definition}

\newtheorem{Lemma}[Theorem]{Lemma}

\newtheorem{Hypothesis}{Hypothesis}

\numberwithin{equation}{section}

\begin{document}
\renewcommand{\figurename}{Fig.1}

\def\le{\left}
\def\r{\right}
\def\cost{\mbox{const}}
\def\a{\alpha}
\def\d{\delta}
\def\ph{\varphi}
\def\e{\epsilon}
\def\la{\lambda}
\def\si{\sigma}
\def\La{\Lambda}
\def\B{{\cal B}}
\def\A{{\mathcal A}}
\def\L{{\mathcal L}}
\def\O{{\mathcal O}}
\def\bO{\overline{{\mathcal O}}}
\def\F{{\mathcal F}}
\def\K{{\mathcal K}}
\def\H{{\mathcal H}}
\def\D{{\mathcal D}}
\def\C{{\mathcal C}}
\def\M{{\mathcal M}}
\def\N{{\mathcal N}}
\def\G{{\mathcal G}}
\def\T{{\mathcal T}}
\def\R{{\mathcal R}}
\def\I{{\mathcal I}}

\def\bw{\overline{W}}
\def\phin{\|\varphi\|_{0}}
\def\s0t{\sup_{t \in [0,T]}}
\def\lt{\lim_{t\rightarrow 0}}
\def\iot{\int_{0}^{t}}
\def\ioi{\int_0^{+\infty}}
\def\ds{\displaystyle}
\def\pag{\vfill\eject}
\def\fine{\par\vfill\supereject\end}
\def\acapo{\hfill\break}

\def\beq{\begin{equation}}
\def\eeq{\end{equation}}
\def\barr{\begin{array}}
\def\earr{\end{array}}
\def\vs{\vspace{.1mm}   \\}
\def\rd{\reals\,^{d}}
\def\rn{\reals\,^{n}}
\def\rr{\reals\,^{r}}
\def\bD{\overline{{\mathcal D}}}
\newcommand{\dimo}{\hfill \break {\bf Proof - }}
\newcommand{\nat}{\mathbb N}
\newcommand{\E}{\mathbb E}
\newcommand{\Pro}{\mathbb P}
\newcommand{\com}{{\scriptstyle \circ}}
\newcommand{\reals}{\mathbb R}

\title{Large deviations for the Langevin equation with strong damping}

\author{Sandra Cerrai\thanks{Partially supported by the NSF grant DMS 1407615.},\  \ Mark Freidlin\thanks{Partially supported by the NSF grant DMS 1411866.}\\
\\
Department of Mathematics\\
University of Maryland\\
College Park, MD 20742\\
USA}
\date{}

\maketitle

\begin{abstract} 
We study large deviations in the Langevin dynamics, with damping of order $\e^{-1}$ and noise of order $1$, as $\e\downarrow 0$. The damping coefficient is assumed to be state dependent. We proceed first with a change of time and then, we use a weak convergence approach to large deviations and their equivalent formulation in terms of the Laplace principle, to determine the good action functional.

Some applications of these results to the exit problem from a domain and to the wave front propagation for a suitable class of reaction diffusion equations are considered.

\end{abstract}

\section{Introduction}
\label{sec:1}
For every $\e>0$, let us consider the Langevin equation
\begin{equation}
\label{eq1}
\le\{\begin{array}{l}
\ds{\ddot{q}^{\,\e}(t)=b(q^\e(t))-\frac {\a(q^\e(t))}\e\dot{q}^{\,\e}(t)+\si(q^\e(t))\dot{B}(t),}\\
\vs
\ds{q^\e(0)=q \in\,\reals^{d},\ \ \ \ \dot{q}^\e(0)=p \in\,\reals^{d}.}
\end{array}
\r.
\end{equation}
Here $B(t)$ is a $r$-dimensional standard Wiener process, defined on some complete stochastic basis $(\Omega,\mathcal{F},\{\mathcal{F}_t\},\mathbb{P})$.  In what follows, we shall assume that  $b$ is Lipschitz continuous and  $\a$ and $\si$ are bounded and continuously differentiable, with bounded derivative. Moreover,  $\si$ is invertible and there exist two constants  $0<\a_0<\a_1$ such that $a_0\leq \a(q)\leq \a_1$, for all $q \in\,\reals^d$.
Equation \eqref{eq1} can be rewritten as the following system in $\reals^{2d}$
\[\le\{\begin{array}{ll}
\ds{\dot{q^\e}(t)=p^\e(t),}  &  \ds{q^\e(0)=q \in\,\reals^d,}\\
\vs
\ds{\dot{p^\e}(t)=b(q^\e(t))-\frac {\a(q^\e(t))}\e\dot{q}^{\e}(t)+\si(q^\e(t))\dot{B}(t),}  &  \ds{p^\e(0)=p\in\,\reals^d,}
\end{array}\r.\]
and, due to our assumptions on the coefficients, for any $\e>0$, $T>0$ and $k\geq 1$, the system above admits a unique solution $z^\e=(q^\e,p^\e) \in\,L^k(\Omega,C([0,T];\reals^{2d}))$, which is a Markov process. 

\medskip

Now, if we do a change of time and  define
$q_\e(t):=q^\e(t/\e)$, $t\geq 0$,
we have 
\begin{equation}
\label{eq2intro}
\le\{\begin{array}{l}
\ds{\e^2 \ddot{q_\e}(t)=b(q_\e(t))-\a(q_\e(t))\dot{q}_{\e}(t)+\sqrt{\e}\,\si(q_\e(t))\dot{w}(t),}\\
\vs
\ds{q_\e(0)=q \in\,\reals^d,\ \ \ \ \dot{q_\e}(0)=\frac p\e \in\,\reals^d,}
\end{array}\r.\end{equation}
where $w(t)=\sqrt{\e}B(t/\e)$, $t\geq0$, is another $\reals^r$-valued Wiener process, defined on the same stochastic basis $(\Omega,\mathcal{F},\{\mathcal{F}_t\},\mathbb{P})$.

\bigskip

In the present paper, we are interested in studying the large deviation principle for equation \eqref{eq2intro}, as $\e\downarrow 0$. Namely, we want to prove that the family $\{q_\e\}_{\e>0}$ satisfies a large deviation principle in $C([0,T];H)$, with the same action functional $I$ and the same normalizing factor $\e$ that describe the large deviation principle for the first order equation
\begin{equation}
\label{cf450}
\dot{g_\e}(t)=\frac{b(g_\e(t))}{\a(g_\e(t))}+\sqrt{\e}\,\frac{\si(g_\e(t))}{\a(g_\e(t))}\dot{w}(t),\ \ \ \ g_\e(0)=q \in\,\reals^d.\end{equation}

In particular, as shown in Section \ref{sec4}, this implies that the asymptotic behavior of the exit time from a basin of attraction for the over damped Langevin dynamics\eqref{eq1} can be described by the quasi potential $V$ associated with $I$, as well as the asymptotic behavior of the solutions of the degenerate parabolic and elliptic problems associated with the  Langevin dynamics.

Moreover, in Section \ref{sec4}, we will show how these results allow to prove that in reaction-diffusion equations with non-linearities of KPP type, where the transport is described by the Langevin dynamics itself, the interface separating
 the areas where $u^\e$ is close to $1$ and to $0$, as $\e\downarrow 0$, is given in terms of the action functional $I$, as in the classical case, when the vanishing mass approximation is considered.

 \bigskip
 
 In \cite{friedlin} and \cite{cf}, the system 
 \begin{equation}
 \label{cf-intro}
 \le\{\begin{array}{l}
\ds{\mu \ddot{q_{\mu,\e}}(t)=b(q_{\mu,\e}(t))-\a(q_{\mu,\e}(t))\dot{q}_{\mu,\e}(t)+\sqrt{\e}\,\si(q_{\mu,\e}(t))\dot{w}(t),}\\
\vs
\ds{q_{\mu,\e}(0)=q \in\,\reals^d,\ \ \ \ \dot{q_{\mu,\e}}(0)=\frac p\e \in\,\reals^d,}
\end{array}\r.\end{equation}
for $0<\mu, \e<<1$, has been studied, under the crucial assumption that the friction coefficient $\a$ is  independent of $q$.

It has been proven that, in this case, the so-called Kramers-Smoluchowski approximation holds, that is for any fixed $\e>0$ the solution $q_{\mu,\e}$ of system \eqref{cf-intro} converges  in $L^2(\Omega;C([0,T];\reals^d))$, as $\mu\downarrow 0$, to $g_\e$, the solution of the first order equation \eqref{cf450}.
Moreover, it has been proven that, if $V_\mu(q,p)$ is the quasi-potential associated with the family $\{q_{\mu,\e}\}_{\e>0}$, for $\mu>0$ fixed, then
\[\lim_{\mu\to 0}\,\inf_{p \in\,\reals^d}V_\mu(q,p)=V(q),\]
where $V$ is the quasi-potential associated with the action-functional $I$.

In \cite{fh},  equation \eqref{cf-intro} with  non constant friction $\a$ has been considered and it has been shown that in this case the situation is considerably more delicate. Actually, the limit of $q_{\mu,\e}$ to $g_\e$ has only been proven via a previous regularization of the noise, which has led to the convergence of $q_{\mu,\e}$ to the solution $\tilde{g}_\e$ of the first order equation with Stratonovich integral.

Finally, we would like to mention that in the recent  paper \cite{lyv}, by Lyv and Roberts,  an analogous problem has been studied for the stochastic damped wave equation in a bounded regular domain $D\subset \reals^d$, with $d=1,2,3$,
\[\le\{\begin{array}{l}
\ds{\e\,\frac{\partial^2 u(t,x)}{\partial t^2}=\Delta u(t,x)+f(u(t,x))-\frac{\partial u(t,x)}{\partial t}+\e^\a\,\frac{\partial w(t,x)}{\partial t}}\\
\vs
\ds{u(t,x)=0,\ \ \ x \in\,\partial D,\ \ \ u(0,x)=u_0(x),\ \ \ \ \frac{\partial u(0,x)}{\partial t}=v_0,(x)}
\end{array}\r.\]
where $\e>0$ is a small parameter, the friction coefficient  is constant ($\a=1$), $w(t,x)$ is a  smooth cylindrical Wiener process and $f$ is a cubic non-linearity. 
By using the weak convergence approach, the authors show that the family $\{u_\e\}_{\e>0}$ satisfies a large deviation principle in $C([0,T];L^2(D))$, with normalizing factor $\e^{2\a}$ and  the same action functional that describes the large deviation principle for the stochastic parabolic equation.

\medskip

As mentioned above, in the present paper we are  dealing with the case of non-constant friction $\a$ and $\mu=\e^2$. Dealing with a non-constant friction coefficient turns out to be important in applications, as it allows to describes new effects in reaction-diffusion equations and exit problems (see section \ref{sec4}). Here, we will study the large deviation principle for equation\eqref{eq2intro} by using the approach of weak convergence (see \cite{de} and \cite{db}) and we will show the validity of the Laplace principle, which, together with the compactness of level sets, is equivalent to the large deviation principle. 

At this point, it is worth mentioning  that one major difficulty here is handling  the integral
\[\int_0^t \exp\le(-\int_s^t\a(q_\e(r))\,dr\r)\si(q_\e(s))\,dw(s),\]
and proving that it converges to zero, as $\e\downarrow 0$, in $L^1(\Omega;C([0,T];\reals^d))$.
Actually, as $\a$ is non-constant, the integral above cannot be interpreted as an It\^o's integral and in our estimates we cannot use It\^o's isometry.
Nevertheless, due to the regularity of $q_\e(t)$, we can consider the integral above as a pathwise integral, and with appropriate integrations by parts, we can get the estimates required to prove the Laplace principle.

\section{The problem and the method}

We are dealing here with the equation
\begin{equation}
\label{eq2}
\le\{\begin{array}{l}
\ds{\e^2 \ddot{q_\e}(t)=b(q_\e(t))-\a(q_\e(t))\dot{q}_{\e}(t)+\sqrt{\e}\,\si(q_\e(t))\dot{w}(t),}\\
\vs
\ds{q_\e(0)=q \in\,\reals^d,\ \ \ \ \dot{q_\e}(0)=\frac p\e \in\,\reals^d,}
\end{array}\r.\end{equation}
Here $w(t)$, $t \geq 0$,  is a $r$-dimensional Brownian motion and the coefficients $b$, $\si$ and $\a$ satisfy the following conditions.

\begin{Hypothesis}
\label{H1}
\begin{enumerate}
\item
The mapping $b:\reals^d\to\reals^d$ is Lipschitz-continuous and the mapping $\si:\reals^d\to {\mathcal L}(\reals^r,\reals^d)$ is continuously differentiable and    bounded, together with its derivative. Moreover, the matrix $\si(q)$ is invertible, for any $q \in\,\reals^d$, and $\si^{-1}:\reals^d\to{\cal L}(\reals^r,\reals^d)$ is bounded. 
\item The mapping $\a:\reals^d\to\reals$ belongs to $C^1_b(\reals^d)$ and 
\begin{equation}
\label{cf11}
\inf_{x \in\,\reals^d}\a(x)=:\a_0>0.
\end{equation}
\end{enumerate}
\end{Hypothesis}
In view of the conditions on the coefficients $\a$, $b$ and $\si$ assumed in Hypothesis \ref{H1}, for every fixed $\e>0$, equation \eqref{eq2} admits a unique solution $z_\e=(q_\e,p_\e) \in\,L^k(0,T;\reals^d)$, with $T>0$ and $k \geq 1$.

Now, for any predictable process $u$ taking values in  $L^2([0,T];\reals^r)$, we introduce the problem 
\begin{equation}
\label{eq3bis}
\dot{g}^{\,u}(t)=\frac{b(g^u(t))}{\a(g^u(t))}+\frac{\si(g^u(t))}{\a(g^u(t))}u(t),\ \ \ \ g^u(0)=q \in\,\reals^d.
\end{equation}
The existence and uniqueness of a pathwise solution $g^u$ to problem \eqref{eq3bis} in $C([0,T];\reals^d)$ is an immediate consequence of the  conditions on the coefficients $b$, $\si$ and $\a$ that we have assumed in Hypothesis \ref{H1}.

In what follows, we shall denote by ${\cal G}$ the mapping
\[{\cal G}:L^2([0,T];\reals^r)\to C([0,T];\reals^d),\ \ \ \ u\mapsto {\cal G}(u)=g^u.\]
 Moreover, for any $f \in\,C([0,T];\reals^d)$ we shall define
 \[I(f)=\frac 12\inf\le\{\,\int_0^T|u(t)|^2\,dt\,:\ f= {\cal G}(u),\ u \in\,L^2([0,T];\reals^r)\,\r\},\]
 with the usual convention $\inf \emptyset=+\infty$. This means that 
 \begin{equation}
 \label{cf41}
 I(f)=\frac 12\int_0^T\le|\a(f(s))\si^{-1}(f(s))\le(\dot{f}(s)-\frac{b(f(s))}{\a(f(s))}\r)\r|^2\,ds,\end{equation}
 for all $f \in\,W^{1,2}(0,T;\reals^d)$.
 
 If we denote by $g_\e$ the solution of the stochastic equation
 \begin{equation}
\label{eq3bis-stoc}
\dot{g_\e}(t)=\frac{b(g_\e(t))}{\a(g_\e(t))}+\sqrt{\e}\,\frac{\si(g_\e(t))}{\a(g_\e(t))}\dot{w}(t),\ \ \ \ g_\e(0)=q \in\,\reals^d,
\end{equation}
we have that $I$ is the large deviation action functional for the family $\{g_\e\}_{\e>0}$ in the space of continuous trajectories $C([0,T];\reals^d)$ (for a proof see e.g. \cite{fw}). This means that the level sets $\{I(f)\leq c\}$ are compact in $C([0,T];\reals^d)$, for any $c>0$, and for any closed subset $F\subset C([0,T];\reals^d)$ and any open set $G \subset C([0,T];\reals^d)$ it holds
\[\begin{array}{l}
\ds{\limsup_{\e\to 0^+}\e\log \mathbb{P}(g_\e \in\,F)\leq -I(F),}\\
\vs
\ds{\liminf_{\e\to 0^+}\e\log \mathbb{P}(g_\e \in\,G)\geq -I(G),}
\end{array}\]
where,  for any subset $A\subset C([0,T];\reals^d)$, we have denoted  
\[I(A)=\inf_{f \in\,A}I(f).\]

The main result of the present paper is to prove that in fact the family of solutions $q_\e$ of equation \eqref{eq2intro} satisfies a large deviation principle with the same action functional $I$ that describes the large deviation principle for  the family of solutions $g_\e$ of equation \eqref{eq3bis-stoc}. And, due to the fact that $q^\e(t)=q_\e(\e t)$, $t\geq 0$, this allows to describe the  behavior of the over damped Langevin dynamics \eqref{eq1} (see Section \ref{sec4} for all details).

\begin{Theorem}
\label{teo1}
Under Hypothesis \ref{H1}, the family of probability measures $\{\mathcal{L}(q_\e)\}_{\e>0}$, in the space of continuous paths $C([0,T];\reals^d)$, satisfies a large deviation principle with action functional $I$.
\end{Theorem} 
In order to prove Theorem \ref{teo1}, we follow the weak convergence approach, as developed in \cite{de}, (see also \cite{db}). To this purpose, we need to introduce some notations.
We  denote by $\mathcal{P}_T$ the set of predictable processes in $L^2(\Omega\times[0,T];\reals^r)$, and for any $T>0$ and $\gamma>0$, we define the sets
\[\begin{array}{l}
\ds{\mathcal{S}^\gamma_T=\le\{ f \in\,L^2(0,T;\reals^d)\,:\,\int_0^T|f(s)|^2\,ds\leq \gamma \r\}}\\
\vs
\ds{\mathcal{A}^\gamma_T=\le\{ u \in\,\mathcal{P}_T\,:\,u \in\,\mathcal{S}^\gamma_T,\ \mathbb{P}-\text{a.s.} \r\}.}
\end{array}\]

Next, for any predictable process $u$ taking values in $L^2([0,T];\reals^r)$, we denote by $q^u_\e(t)$ the solution of the problem
\begin{equation}
\label{eq2}
\le\{\begin{array}{l}
\ds{\e^2 \ddot{q}^{\,u}_\e(t)=b(q^u_\e(t))-\a(q^u_\e(t))\dot{q}^{\,u}_{\e}(t)+\sqrt{\e}\,\si(q^u_\e(t))\dot{w}(t)+\si(q^u_\e(t))u(t),}\\
\vs
\ds{q^u_\e(0)=q \in\,\reals^d,\ \ \ \ \dot{q}^{\,u}_\e(0)=\frac p\e \in\,\reals^d.}
\end{array}
\r.
\end{equation}
As well known, for any fixed $\e>0$ and for any $T>0$ and $k \geq 1$, this equation admits a unique solution $q^u_\e$ in $L^k(\Omega;C([0,T];\reals^d))$.

By proceeding as in the proof of \cite[Theorem 4.3]{db}, the following result can be proven.
\begin{Theorem}
\label{teo-bd}
Let $\{u_\e\}_{\e>0}$ be a family of processes in ${\mathcal S}^\gamma_T$ that converge in distribution, as $\e\downarrow 0$, to some $u \in\,{\mathcal S}^\gamma_T$, as random variables taking values in the space $L^2(0,T;\reals^d)$, endowed with the weak topology. 

If the sequence $\{q^{u_\e}_\e\}_{\e>0}$ converges in distribution to $g^u$, as $\e\downarrow 0$, in the space of continuous paths $C([0,T];\reals^d)$, then the family $\{{\mathcal L}(q_\e)\}_{\e>0}$ satisfies a large deviation principle in $C([0,T];\reals^d)$, with action functional $I$.
\end{Theorem}

Actually, as shown in \cite{db}, the convergence of $q^{u_\e}_\e$ to $g^u$ implies the validity of the Laplace principle with rate functional $I$. This means that, for any continuous mapping $\Lambda:C([0,T];\reals^d)\to \reals$ it holds
\[\lim_{\e\to 0}-\e\log\E\,\exp\le(-\frac 1\e\,\Lambda(q_\e)\r)=\inf_{f \in\,C([0,T];\reals^d)}\le(\,\Lambda(f)+I(f)\,\r).\]
And, as the level sets of $I$ are compact, this is equivalent to say that $\{{\mathcal L}(q_\e)\}_{\e>0}$ satisfies a large deviation principle in $C([0,T];\reals^d)$, with action functional $I$.

\section{Proof of Theorem \ref{teo1}}
\label{sec:2}

As we have seen in the previous section, in order to prove Theorem \ref{teo1}, we have to show that  
if $\{u_\e\}_{\e>0}$ is a family of processes in ${\mathcal S}^\gamma_T$ that converge in distribution, as $\e\downarrow 0$, to some $u \in\,{\mathcal S}^\gamma_T$, as random variables taking values in the space $L^2(0,T;\reals^d)$, endowed with the weak topology, then
the sequence $\{q^{u_\e}_\e\}_{\e>0}$ converges in distribution to $g^u$, as $\e\downarrow 0$, in the space $C([0,T];\reals^d)$.

In view of the Skorohod representation theorem, we can  rephrase such a condition in the following way. 
 On some probability space $(\bar{\Omega},\bar{{\mathcal F}},\bar{{\mathbb P}})$, consider a Brownian motion $\bar{w}_t$, $t\geq 0$,  along with the corresponding natural filtration $\{\bar{{\mathcal F}}_t\}_{t\geq 0}$. Moreover, consider a family of $\{\bar{{\mathcal F}}_t\}$-predictable processes $\{\bar{u}_\e,\ \bar{u}\}_{\e>0}$ in $L^2(\bar{\Omega}\times [0,T];\reals^d)$, taking values in ${\mathcal S}^\gamma_T$, $\bar{{\mathbb P}}$-a.s., such that the joint law of $(\bar{u}^\e, \bar{u},\bar{w})$, under $\bar{{\mathbb P}}$, coincides with the  joint law of $(u^\e, u,w)$, under ${\mathbb P}$, and such that
 \begin{equation}
 \label{cf70}
 \lim_{\e\to 0} \bar{u}_\e=\bar{u},\ \ \ \ \bar{{\mathbb P}}-\text{a.s.}\end{equation}
 as $L^2(0,T;\reals)$-valued random variables, endowed with the weak topology. Let $\bar{q}^{\,\bar{u}_\e}_\e$ be the solution of a problem analogous to  \eqref{eq2}, with $u$ and $w$ replaced respectively by $\bar{u}_\e$ and $\bar{w}$.

Then, we have to prove that 
\[\lim_{\e\to 0}\bar{q}^{\,\bar{u}_\e}_\e=g^{\bar{u}},\ \ \ \ \ \bar{{\mathbb P}}-\text{a.s.}\]
in $C([0,T];\reals^d)$.
In fact, we will prove more. Actually, we will show that 
\begin{equation}
\label{cf75}
\lim_{\e\to 0}\,\bar{\E}\sup_{t \in\,[0,T]}|\bar{q}^{\bar{u}_\e}_\e(t)-g^{\bar{u}}(t)|=0.\end{equation}

In order to prove \eqref{cf75}, we will need some preliminary  estimates. 
For any $\e>0$, we define
 the process 
\begin{equation}
\label{cf100}
H_\e(t)=\sqrt{\e}\,e^{-A_\e(t)}\int_0^t e^{A_\e(s)}\si(q^{u}_\e(s))\,dw(s),\ \ \ \ t\geq 0.
\end{equation}

\begin{Lemma}
Under Hypothesis \ref{H1}, for any $T>0$, $k\geq 1$ and $\gamma>0$, there exists $\e_0>0$ such that for any $u \in\,S^\gamma_T$ and $\e \in\,(0,\e_0]$  
\begin{equation}
\label{cf106}
\sup_{s\leq t}\E\,|H_\e(t)|^k\leq c_{k,\gamma}(T)(|q|^k+|p|^k+1)\e^{\frac {3k}2}+c_k\,\e^{\frac k2}\,t^{\frac k2}e^{-\frac{k\a_0 t}{\e^2}}.
\end{equation}
Moreover, we have
\begin{equation}
\label{cf121}
\E\sup_{t \in\,[0,T]}|H_\e(t)|\leq \sqrt{\e}\,c_\gamma(T)(1+|q|+|p|).
\end{equation}

\end{Lemma}

\begin{proof}
Equation \eqref{eq2} can be rewritten as the system
\[\le\{\begin{array}{l}
\ds{\dot{q}^{\,u}_\e(t)=p^{u}_\e(t),\ \ \ \ \ q^{u}_\e(0)=q}\\
 \vs
\ds{\e^2 \dot{p}^{\,u}_\e(t)=b(q^{u}_\e(t))-\a(q^{u}_\e(t))p^{\,{u}}_{\e}(t)+\sqrt{\e}\,\si(q^{u}_\e(t))\dot{w}(t)+\si(q^{u}_\e(t))u(t),\ \ \ \ \ p^{\,u}_\e(0)=\frac p\e.}
\end{array}
\r.
\]
Thus, if for any $0\leq s\leq t$ and $\e>0$ we define
\[A_{\e}(t,s):=\frac 1{\e^2}\int_s^t\a(q^{u}_\e(r))\,dr,\ \ \ \ A_{\e}(t):=A_{\e}(t,0),\]
we have
\begin{equation}
\label{cf13}
\begin{array}{ll}
\ds{p^{u}_\e(t)=}  &  \ds{\frac 1\e e^{-A_{\e}(t)}p+\frac 1{\e^2}\int_0^t e^{-A_{\e}(t,s)}b(q^{u}_\e(s))\,ds}\\
&  \vs
&\ds{+\frac 1{\e^2}\int_0^t e^{-A_{\e}(t,s)}\si(q^{u}_\e(s))\,u_\e(s)\,ds+\frac 1{\e^2} H_{\e}(t).}
\end{array}\end{equation}
Integrating with respect to $t$, this yields
\begin{equation}
\label{cf82}
\begin{array}{ll}
\ds{q^{u}_\e(t)=}  &   \ds{q+\frac 1\e\int_0^te^{-A_{\e}(s)}p\,ds+\frac 1{\e^2}\int_0^t \int_0^s e^{-A_{\e}(s,r)}b(q^{u}_\e(r))\,dr\,ds}\\
&  \vs
& \ds{+\frac 1{\e^2}\int_0^t \int_0^s e^{-A_{\e}(s,r)}\si(q^{u}_\e(r))\,u_\e(r)\,dr\,ds+\frac 1{\e^2}\int_0^t H_{\e}(s)\,ds.}
\end{array}
\end{equation}

Thanks to the Young inequality, this implies that for any $t \in\,[0,T]$
\[\begin{array}{l}
\ds{ |q^{u}_\e(t)|\leq |q|+\e\,|p|+c\int_0^t (1+|q^{u}_\e(s)|)\,ds+\int_0^t |u_\e(s)|\,ds+\frac 1{\e^2}\int_0^t |H_{\e}(s)|\,ds}\\
\vs
\ds{\leq  c_\gamma(T)(|q|+\e\,|p|+1) 
+\frac 1{\e^2}\int_0^t |H_{\e}(s)|\,ds+\int_0^t |q^{u}_\e(s)|\,ds,}\end{array}\]
and from the Gronwall lemma we can conclude that
\[|q^{u}_\e(t)|\leq c_\gamma(T)\le(1+|q|+|p|\r)+c(T)\frac 1{\e^2}\int_0^t |H_{\e}(s)|\,ds.\]
This implies that for any $k\geq 1$ 
 \begin{equation}
\label{cf103}
|q^{u}_\e(t)|^ k\leq c_{k,\gamma}(T)(|q|^k+|p|^k+1)+c_{k,\gamma}(T)\e^{-2 k}\int_0^t |H_{\e}(s)|^k\,ds,\ \ \ \ \e \in\,(0,1].
\end{equation}
Now, due to \eqref{cf13}, we have
\[\begin{array}{l}
\ds{|p^{u_\e}(t)|\leq \frac 1\e e^{-\frac{\a_0 t}{\e^2}}|p|+\frac 1{\e^2}\int_0^t e^{-\frac{\a_0(t-s)}{\e^2}}\le(1+|q^{u}_\e(s)|\r)\,ds}\\
\vs
\ds{+\frac 1{\e^2}\int_0^t e^{-\frac{\a_0(t-s)}{\e^2}}|u_\e(s)|\,ds+\frac1{\e^2}|H_\e(t)|,}
\end{array}\]
so that, thanks to \eqref{cf103}, for any $\e \in\,(0,1]$ we get
\begin{equation}
\label{cf104}
\begin{array}{l}
\ds{|p^{u}_{\e}(t)|\leq \frac 1\e e^{-\frac{\a_0 t}{\e^2}}|p|+c_{\gamma}(T)(|q|+|p|+1)+\frac 1{\e^2}\int_0^t e^{-\frac{\a_0(t-s)}{\e^2}}|u_\e(s)|\,ds+c(T)\frac1{\e^2}|H_\e(t)|.}
\end{array}\end{equation}

As well known, if $f \in\,C^1([0,t])$ and $g \in\,C([0,t])$,  then the Stiltjies integral
\[\int_0^t f(s)dg(s),\ \ \ \ t\geq 0,\]
is well defined and, if  $g(0)=0$, the following integration by parts formula holds
\begin{equation}
\label{cf101}
\int_0^t f(s)dg(s)=\int_0^t\le(g(t)-g(s)\r) h^\prime(s)\,ds+g(t)h(0),\ \ \ \ t\geq 0.
\end{equation}

Now, the mapping
\[[0,+\infty)\to {\mathcal L}(\reals^r,\reals^d),\ \ \ \ s\mapsto e^{A_{\e}(s)}\si(q^{u}_\e(s)),\]
is differentiable, $\mathbb{P}$-a.s.,  so that the stochastic integral in \eqref{cf100} is in fact a pathwise integral. In particular, we can apply formula \eqref{cf101}, with
\[h(s)=e^{A_\e(s)}\si(q^{u}_\e(s)),\ \ \ \ g(s)=w(s),\]
and we get
\begin{equation}
\label{cf300}
\begin{array}{ll}
\ds{H_\e(t)=}  &  \ds{\sqrt{\e}\int_0^t (w(t)-w(s))\,e^{-A_\e(t,s)}\le(\frac{\a(q^{u}_\e(s))}{\e^2}+\si^\prime(q^{u}_\e(s))p^{u}_\e(s)\r)\,ds}\\
& \vs
& \ds{+\sqrt{\e}w(t)e^{-A_\e(t)}\si(q).}
\end{array}\end{equation}
Thanks to \eqref{cf104}, this yields for any $\e \in\,(0,1]$
\[\begin{array}{l}
\ds{|H_\e(t)|\leq c\,\sqrt{\e}\int_0^t|w(t)-w(s)|\frac{e^{-\frac{\a_0(t-s)}{\e^2}}}{\e^2}\le(1+\e^2 |p^u_\e(s)|\r)\,ds+c\,\sqrt{\e}\,|w(t)|e^{-\frac{\a_0 t}{\e^2}}}\\
\vs
\ds{\leq c_\gamma(T)(|q|+|p|+1)\sqrt{\e}\int_0^{\frac{t}{\e^2}}|w(t)-w(t-\e^2 s)|e^{-\a_0 s}\,ds}\\
\vs
\ds{+\sqrt{\e}\,c_\gamma(T)\int_0^{\frac{t}{\e^2}}|w(t)-w(t-\e^2 s)|e^{-\a_0 s}|H_\e(t-\e^2 s)|\,ds+c\,\sqrt{\e}\,|w(t)|e^{-\frac{\a_0 t}{\e^2}},}
\end{array}\]
and hence, for any $k \geq 1$, we have
\[\begin{array}{l}
\ds{|H_\e(t)|^k\leq c_{k,\gamma}(T)(|q|^k+|p|^k+1)\e^{\frac {k}2}\,\int_0^{\frac{t}{\e^2}}|w(t)-w(t-\e^2 s)|^ke^{-\a_0 s}\,ds}\\
\vs
\ds{+\e^{\frac k2}\,c_{k,\gamma}(T)\int_0^{\frac{t}{\e^2}}|w(t)-w(t-\e^2 s)|^ke^{-\a_0 s}|H_\e(t-\e^2 s)|^k\,ds+c_k\,\e^{\frac k2}\,|w(t)|^ke^{-\frac{k\a_0 t}{\e^2}}.}
\end{array}\]
By taking the expectation, due to the independence of $|w(t)-w(t-\e^2 s)|$ with $|H_\e(t-\e^2 s)|$ and $\int_0^{t-\e^2s}|H_\e(r)|^k\,dr$, this implies that for any $\e \in\,(0,1]$
\[\begin{array}{l}
\ds{\E\,|H_\e(t)|^k\leq c_{k,\gamma}(T)(|q|^k+|p|^k+1)\e^{\frac {3k}2}\int_0^{\frac{t}{\e^2}} s^{\frac k2}e^{-\a_0 s}\,ds}\\
\vs
\ds{+\e^{\frac {3k}2}\,c_{k,\gamma}(T)\int_0^{\frac{t}{\e^2}}s^{\frac k2}e^{-\a_0 s}\E|H_\e(t-\e^2 s)|^k\,ds+c_k\,\e^{\frac k2}\,t^{\frac k2}e^{-\frac{k\a_0 t}{\e^2}}}\\
\vs
\ds{\leq c_{k,\gamma}(T)(|q|^k+|p|^k+1)\e^{\frac {3k}2}+c_k\,\e^{\frac k2}\,t^{\frac k2}e^{-\frac{k\a_0 t}{\e^2}}+\e^{\frac {3k}2}c_{k,\gamma}(T)\sup_{s\leq t}\E\,|H_\e(s)|^k.}
\end{array}\]
Therefore, if we pick $\e_0 \in\,(0,1]$ such that
\[\e^{\frac {3k}2}c_{k,\gamma}(T)<\frac 12,\]
we get \eqref{cf106}.

Now, let us prove \eqref{cf121}. From \eqref{cf300}, we have
\[\begin{array}{l}
\ds{|H_\e(t)|\leq \sqrt{\e}\,c\,\sup_{t \in\,[0,T]}|w(t)|
\le(1+\int_0^t e^{-\frac{\a_0(t-2)}{\e^2}}|p^{u_\e}_\e(s)|\,ds\r)}\\
\vs
\ds{\leq \sqrt{\e}\,c\,\sup_{t \in\,[0,T]}|w(t)|
\le(1+\e \le(\int_0^t |p^{u_\e}_\e(s)|^2\,ds\r)^{\frac 12}\r),}
\end{array}\]
and hence
\[\E\sup_{t \in\,[0,T]}|H_\e(t)|\leq  \sqrt{\e}\,c(T)\le(1+\e \le(\E\int_0^t |p^{u_\e}_\e(s)|^2\,ds\r)^{\frac 12}\r).\]
Thanks to \eqref{cf104}, as a consequence of the Young inequality, we get
\begin{equation}
\label{cf122}
\int_0^t |p^{u_\e}_\e(s)|^2\,ds\leq c_\gamma(T)(1+|q|^2+|p|^2)+\frac 1{\e^{4}}\,c(T)\int_0^t|H_\e(s)|^2\,ds,\end{equation}
so that
\[\E\sup_{t \in\,[0,T]}|H_\e(t)|\leq  \sqrt{\e}\,c_\gamma(T)(1+|q|+|p|)+\frac 1{\sqrt{\e}}\,c(T)\le(\int_0^t\E\,|H_\e(s)|^2\,ds\r)^{\frac 12}.\]
Therefore, \eqref{cf121} follows from \eqref{cf106}.
\end{proof}

\begin{Lemma}
Under Hypothesis \ref{H1}, for any $T>0$, $k\geq 1$ and $\gamma>0$ there exists $\e_0>0$ such that for any $u \in\,S^\gamma_T$ and $\e \in\,(0,\e_0)$ we have
\begin{equation}
\label{cf110}
\E\,\sup_{t \in\,[0,T]}|q^{u}_\e(t)|^ k\leq c_{k,\gamma}(T)(|q|^k+|p|^k+1)\,\e^{-\frac k2}+c_{k,\gamma}(T)\e^{2 -\frac{3k}2}.
\end{equation}

\begin{proof}
Estimate \eqref{cf110} follows by combining together \eqref{cf106} and \eqref{cf103}.

\end{proof}

\end{Lemma}

Now, we are ready to prove \eqref{cf75}, that, in view of  Theorem  \ref{teo-bd}, implies Theorem \ref{teo1}.
\begin{Theorem}
Let $\{u_\e\}_{\e>0}$ be a family of predictable processes in $\mathcal{S}^\gamma_T$ that converge $\mathbb{P}$-a.s., as $\e\downarrow 0$, to some $u \in\, \mathcal{S}^\gamma_T$, with respect to the weak topology of $L^2(0,T;\reals^d)$. Then,   we have
\begin{equation}
\label{cf5}
\lim_{\e\to 0}\,\E \sup_{t \in\,[0,T]}|q^{u_\e}_\e(t)-g^u(t)|=0.
\end{equation}
\end{Theorem}

\begin{proof}
Integrating by parts in \eqref{cf82}, we obtain
\[q^{u_\e}_\e(t)=q+\int_0^t\frac{b(q^{u_\e}_\e(s))}{\a(q^{u_\e}_\e(s))}\,ds+\int_0^t\frac{\si(q^{u_\e}_\e(s))}{\a(q^{u_\e}_\e(s))}u_\e(s)\,ds+R_{\e}(t),\]
where 
\[\begin{array}{l}
\ds{R_{\e}(t)=\frac p\e\int_0^te^{-A_{\e}(s)}\,ds-\frac 1{\a(q^{u_\e}_\e(t))}\int_0^t e^{-A_{\e}(t,s)}b(q^{u_\e}_\e(s))\,ds+\sqrt{\e}\int_0^t\frac{\si(q^{u_\e}_\e(s))}{\a(q^{u_\e}_\e(s))}\,dw(s)}\\
\vs
\ds{+\int_0^t\le(\int_0^s e^{-A_{\e}(s,r)}b(q^{u_\e}_\e(r))\,dr\r)\frac 1{\a^2(q^{u_\e}_\e(s))}\le<\nabla \a(q^{u_\e}_\e(s)),p^{u_\e}_\e(s)\r>\,ds}\\
\vs
\ds{-\frac{1}{\a(q^{u_\e}_\e(t))}H_{\e}(t)+\int_0^t\frac 1{\a^2(q^{u_\e}_\e(s))}H_{\e}(s)\le<\nabla \a(q^{u_\e}_\e(s)),p^{u_\e}_\e(s)\r>\,ds=:\sum_{k=1}^6I^k_{\e}(t).}
\end{array}\]
This implies that
\begin{equation}
\label{cf1}
\begin{array}{l}
\ds{q^{u_\e}_\e(t)-g^u(t)=\int_0^t\le[\frac{b(q^{u_\e}_\e(s))}{\a(q^{u_\e}_\e(s))}-\frac{b(g^{u}(s))}{\a(g^{u}(s))}\r]\,ds+\int_0^t\le[\frac{\si(q^{u_\e}_\e(s))}{\a(q^{u_\e}_\e(s))}-\frac{\si(g^{u}(s))}{\a(g^{u}(s))}\r]u_\e(s)\,ds}\\
\vs
\ds{
+\int_0^t\frac{\si(g^{u}(s))}{\a(g^{u}(s))}\le[u_\e(s)-u(s)\r]\,ds+R_{\e}(t).}
\end{array}\end{equation}
Due to the Lipschitz-continuity and the boundedness of the functions  $\si$ and $1/\a$, we have that  $\si/\a$ is bounded and Lipschitz continuous. Then,  as $u_\e \in\,{\mathcal S}^\gamma_T$, we obtain
\[\begin{array}{l}
\ds{|q^{u_\e}_\e(t)-g^u(t)|^2}\\
\vs
\ds{\leq c\le|\int_0^t\frac{\si(g^{u}(s))}{\a(g^{u}(s))}\le[u_\e(s)-u(s)\r]\,ds\r|^2+c\,|R_{\e}(t)|^2+c(T)\int_0^t|q^{u_\e}_\e(s)-g^u(s)|^2\,ds}\\
\vs
\ds{+c(T)\int_0^t|q^{u_\e}_\e(s)-g^u(s)|^2\,ds\le(\int_0^t |u_\e(s)|^2\,ds+\sup_{s \in\,[0,t]}|g^u(s)|^2\r)}\\
\vs
\ds{\leq c\le|\int_0^t\frac{\si(g^{u}(s))}{\a(g^{u}(s))}\le[u_\e(s)-u(s)\r]\,ds\r|^2+c\,|R_{\e}(t)|^2+c_\gamma(T)\int_0^t|q^{u_\e}_\e(s)-g^u(s)|^2\,ds.}
\end{array}\]
By the Gronwall lemma, this allows to conclude that
\begin{equation}
\label{cf30}
\begin{array}{l}
\ds{\sup_{t \in\,[0,T]}|q^{u_\e}_\e(t)-g^u(t)|}\\
\vs
\ds{\leq c_\gamma(T)\sup_{t \in\,[0,T]}\le|\int_0^t\frac{\si(g^{u}(s))}{\a(g^{u}(s))}\le[u_\e(s)-u(s)\r]\,ds\r|+c_\gamma(T)\,\sup_{t \in\,[0,T]}|R_{\e}(t)|.}
\end{array}\end{equation}

Now, for any $\e>0$, we define
\[\Gamma_\e(t)=\int_0^t\frac{\si(g^{u}(s))}{\a(g^{u}(s))}\le[u_\e(s)-u(s)\r]\,ds.\]
For any $0<s<t$ we have
\[\Gamma_\e(t)-\Gamma_\e(s)=\int_s^t\frac{\si(g^{u}(r))}{\a(g^{u}(r))}\le[u_\e(r)-u(r)\r]\,dr,\]
so that, as $u_\e$ and $u$ are both in $S^\gamma_T$, 
\[|\Gamma_\e(t)-\Gamma_\e(s)|\leq c_\gamma \sqrt{t-s},\ \ \ \e>0.\]
As $\Gamma_\e(0)=0$, this implies that the family of continuous functions is $\{\Gamma_\e\}_{\e>0}$ is equibounded and equicontinuous, so that, by the Ascoli-Arzel\`a theorem, there exists $\e_n\downarrow 0$ and $v \in\,C([0,T];\reals^d)$ such that
\[\lim_{n\to 0}\sup_{t \in\,[0,T]}|\Gamma_{\e_n}(t)-v(t)|=0,\ \ \ \ \mathbb{P}-\text{a.s.}\]

On the other hand, as \eqref{cf70} holds, for any $h \in\,\reals^d$ we have
\[\lim_{\e\to 0}\le<\Gamma_{\e}(t),h\r>=\lim_{\e\to 0}\le<u_\e-u,\frac{\si(g^{u}(\cdot))}{\a(g^{u}(\cdot))}h\r>_{L^2(0,T;\reals^d)}=0,\]
so that we can conclude that $v=0$ and 
\[\lim_{\e\to 0}\,\E\sup_{t\in\,[0,T]}|\Gamma_\e(t)|=0.\]

Thanks to \eqref{cf30}, this implies that 
\[\limsup_{\e\to 0}\E\sup_{t \in\,[0,T]}|q^{u_\e}_\e(t)-g^u(t)|\leq c\,\limsup_{\e\,\to 0}\,\E\sup_{t \in\,[0,T]}|R_{\e}(t)|,\]
so that \eqref{cf5} follows if we show that
\begin{equation}
\label{cf10}
\lim_{\e\to 0}\,\E\sup_{t \in\,[0,T]}|R_{\e}(t)|=0.
\end{equation}

We have 
\begin{equation}
\label{cf21}
\begin{array}{l}
\ds{|I^1_{\e}(t)|=\frac {|p|}\e\le|\int_0^te^{-A_{\e}(s)}\,ds\r|\leq c\,|p|\,\e^{-1}\int_0^te^{-\frac{\a_0 s}{\e^2}}\,ds\leq 
c\,|p|\,\e.}
\end{array}\end{equation}

Moreover
\[\begin{array}{l}
\ds{|I^2_{\e}(t)|=\frac 1{|\a(q^{u_\e}_\e(t))|}\le|\int_0^t e^{-A_{\e}(t,s)}b(q^{u_\e}_\e(s))\,ds\r|}\\
\vs
\ds{\leq c\,\int_0^t e^{-\frac{\a_0(t-s)}{\e^2}}(1+|q^{u_\e}_\e(s)|)\,ds\leq c\,\e^2\le(1+\sup_{t \in\,[0,T]}|q^{u_\e}_\e(t)|\r).}
\end{array}\]
Thanks to \eqref{cf110}, this  implies 
\begin{equation}
\label{cf22}
\E\sup_{t \in\,[0,T]}|I^2_{\e}(t)|\leq c_{\gamma}(T)(|p|+|q|+1)\e^{\frac 32},\ \ \ \ \e \in\,(0,1].
\end{equation}

Next
\begin{equation}
\label{cf23}
\E\sup_{t  \in\,[0,T]}|I^3_{\e}(t)|=\sqrt{\e}\,\E\sup_{t \in\,[0,T]}\le|\int_0^t\frac{\si(q^{u_\e}_\e(s))}{\a(q^{u_\e}_\e(s))}\,dw(s)\r|\leq 
c(T)\,\sqrt{\e}.\end{equation}

Concerning $I^4(t)$, we have
\[\begin{array}{l}
\ds{|I^4_{\e}(t)|=\le|\int_0^t\le(\int_0^s e^{-A_{\e}(s,r)}b(q^{u_\e}_\e(r))\,dr\r)\frac 1{\a^2(q^{u_\e}_\e(s))}\le<\nabla \a(q^{u_\e}_\e(s)),p^{u_\e}_\e(s)\r>\,ds\r|}\\
\vs
\ds{\leq \e^2\,c\,\le(1+\sup_{t \in\,[0,T]}|q^{u_\e}_\e(t)|\r)\int_0^t|p^{u_\e}_\e(s)|\,ds,}
\end{array}\]
so that, due to \eqref{cf110} we obtain
\[
\begin{array}{l}
\ds{\E\sup_{t  \in\,[0,T]}|I^4_{\e}(t)|\leq \e^{2}\,c_{\gamma}(T)(|q|+|p|+1)\,\e^{-\frac 12}\le(\E\int_0^t |p^{u_\e}_\e(s)|^2\,ds\r)^{\frac 12}.}\end{array}
\]
As a consequence of \eqref{cf106} and  \eqref{cf122}, this yields 
\begin{equation}
\label{cf120}
\E\sup_{t  \in\,[0,T]}|I^4_{\e}(t)|\leq \e\,c_{\gamma}(T)(|q|^2+|p|^2+1),\ \ \ \ \e \in\,(0,\e_0].
\end{equation}

Concerning $I^5_\e(t)$,  according to \eqref{cf121} we have
\begin{equation}
\label{cf25}
\E \sup_{t  \in\,[0,T]}|I^5_{\e}(t)|\leq c\,\E \sup_{t  \in\,[0,T]}|H_\e(t)|\leq \sqrt{\e}\,c_\gamma(T)(1+|q|+|p|).
\end{equation}
Finally, it remains to estimate $I^6_{\e}(t)$.
We have
\[\begin{array}{l}
\ds{|I^6_{\e}(t)|=\le|\int_0^t\frac 1{\a^2(q^{u_\e}_\e(s))}H_{\e}(s)\le<\nabla \a(q^{u_\e}_\e(s)),p^{u_\e}_\e(s)\r>\,ds\r|\leq c\int_0^t|H_{\e}(s)||p^{u_\e}_\e(s)|\,ds,}
\end{array}\]
so that
\[\E\sup_{t \in\,[0,T]}|I^6_{\e}(t)|\leq c\le(\int_0^T\E|H_\e(s)|^2\,ds\int_0^T\E|p^{u_\e}_\e(s)|^2\,ds\r)^{\frac 12}.\]
By using \eqref{cf122}, this gives
\[\E\sup_{t \in\,[0,T]}|I^6_{\e}(t)|\leq c_\gamma(T)(1+|q|+|p|)\le(\int_0^T\E|H_\e(s)|^2\,ds\r)^{\frac 12}+\frac 1{\e^2}\int_0^T\E|H_\e(s)|^2\,ds,\]
so that, from \eqref{cf106} we get
\[\E\sup_{t \in\,[0,T]}|I^6_{\e}(t)|\leq \e\,c_\gamma(T)(1+|q|+|p|),\ \ \ \e \in\,(0,\e_0].\]
This, together with \eqref{cf21}, \eqref{cf22}, \eqref{cf23}, \eqref{cf120} and \eqref{cf25}, implies \eqref{cf10} and 
\eqref{cf5} follows.

\end{proof}

\section{Some applications and remarks}
\label{sec4}

Let $G$ be a bounded domain in $\reals^d$, with a smooth boundary $\partial G$. We consider here the exit problem for the process $q^\e(t)$ defined as the solution of equation \eqref{eq1}. For every $\e>0$ we define
\[\tau^\e:=\min\{t\geq 0\,:\,q^\e(t) \notin G\},\ \ \ \ \ \tau_\e:=\min\{t\geq 0\,:\,q_\e(t) \notin G\},\]
where $q_\e(t)=q^\e(t/\e)$ is the solution of equation \eqref{eq2}. It is clear that 
\[\tau^\e=\frac 1\e \tau_\e,\ \ \ \ \ q^\e(\tau^\e)=q_\e(\tau_\e).\]

In what follows, we shall assume that the dynamical system
\begin{equation}
\label{cf40}
\dot{q}(t)=b(q(t)),\ \ \ \ t\geq 0,
\end{equation}
satisfies the following conditions.
\begin{Hypothesis}
\label{H2}
The point $O \in\,G$ is asymptotically stable for the dynamical system \eqref{cf40} and for any initial condition $q \in\,\reals^d$
\[\lim_{t\to \infty} q(t)=O.\]
Moreover, we have
\[\le<b(q),\nu(q)\r>>0,\ \ \ \ \ \ q \in\,\partial G,\]
where $\nu(q)$ is the inward normal vector at $q \in\,\partial G$.
\end{Hypothesis}

Now, we introduce the quasi-potential associated with the action functional $I$ defined in \eqref{cf41}
\[\begin{array}{l}
\ds{V(q)=\inf\,\le\{I(f),\ f \in\,C([0,T];\reals^d),\ f(0)=O,\ f(T)=q,\ T>0\r\}}\\
\vs
\ds{=\frac 12\inf\le\{\int_0^T\le|\a(f(s))\si^{-1}(f(s))\le(\dot{f}(s)-\frac{b(f(s))}{\a(f(s))}\r)\r|^2\,ds,\ f(0)=O,\ f(T)=q,\ T>0\r\}.}
\end{array}\]
It is easy to check that, under our assumptions on $\a(q)$, the quasi-potential $V$ coincides with
\begin{equation}
\label{cf400}
\frac 12\inf\le\{\int_0^T\le|\si^{-1}(f(s))\le(\dot{f}(s)-\a(f(s))b(f(s))\r)\r|^2\,ds,\ f(0)=O,\ f(T)=q,\ T>0\r\}.
\end{equation}

\begin{Theorem}
\label{teo3.1}
Under Hypotheses \ref{H1} and \ref{H2}, for each $q \in\,\{q \in\,G\,:\,V(q)\leq V_0\}$ and  $p \in\,\reals^d$, we have
\begin{equation}
\label{cf42}
\lim_{\e\to 0}\e\log\E_{(q,p)}\tau^\e=\lim_{\e\to 0}\e\log\E_{(q,p)}\tau_\e=V_0,\end{equation}
and
\begin{equation}
\label{cf43}
\lim_{\e\to 0}\e\log\tau^\e=\lim_{\e\to 0}\e\log \tau_\e=V_0,\ \ \ \ \text{in probability},\end{equation}
where  
\[V_0:=\min_{q \in\,\partial G}V(q).\]
Moreover, if the minimum of $V$ on $\partial G$ is achieved at a unique point $q^\star \in\,\partial G$, then
\begin{equation}
\label{cf44}
\lim_{\e\to 0}q^\e(\tau^\e)=\lim_{\e\to 0}q_\e(\tau_\e)=q^\star.
\end{equation}
\end{Theorem}

\begin{proof}
First, note that $q_\e(t)$ is the first component of the $2d$-dimensional Markov process $z_\e(t)=(q_\e(t),p_\e(t))$. Because of the structure of the $p$-component of the drift   of this process and our assumptions on the vector field $b$, starting from $(q,p) \in\,\reals^{2d}$, the trajectory of $z_\e(t)$ spends most of the time in a small neighborhood of the point $q=O$ and $p=0$, with probability close to $1$, as $0<\e<<1$. From time to time, the process $z_\e(t)$ deviates from this point and, as proven in Theorem \ref{teo1},  the deviations of $q_\e(t)$ are governed by the large deviation principle with action functional $I$, defined in \eqref{cf41}. This allows to prove the validity of \eqref{cf42}, \eqref{cf43} and \eqref{cf44} in the same way as Theorems 4.41, 4.42 and 4.2.1 from \cite{fw} are proven. We omit the details.

\end{proof}

As an immediate consequence of \eqref{cf400} and \cite[Theorem 4.3.1]{fw}, we have the following result.
\begin{Theorem}
\label{teo3.2}
Assume $a(q):=\si(q)\si^\star(q)=I$ and 
$\a(q)b(q)=-\nabla U(q)+l(q)$, for any $q \in\,\reals^d$,
for some  smooth function $U:\reals^d\to \reals$  having a unique critical point (a minimum) at $O \in\,\reals^d$ and such that 
\[\le<\nabla U(q),l(q)\r>=0,\ \ \ \ q \in\,\reals^d.\]
Then
\[V(q)=2U(q),\ \ \ \ \ q \in\,\reals^d.\]
\end{Theorem}
From Theorems \ref{teo3.1} and \ref{teo3.2}, it is possible to get a number of results concerning the asymptotic behavior, as $\e\downarrow 0$, of the solutions of the degenerate parabolic and the elliptic problems associated with the differential operator ${\cal L}^\e$ defined by
\[{\cal L}^\e u(q,p)=\frac 12\sum_{i,j=1}^d a_{i,j}(q)\frac{\partial^2 u}{\partial p_i\partial p_j}(q,p)+\le(b(q)-\frac 1\e \a(q)p\r)\cdot\nabla_p u(q,p)+p\cdot \nabla_q u(q,p).\]

\medskip

Assume now that the dynamical system \eqref{cf40} has several asymptotically stable attractors. Assume, for the sake of brevity, that all attractors are just stable equilibriums $O_1$, $O_2$,\ldots,$O_l$. Denote by ${\cal E}$ the set of separatrices separating the basins of these attractors, and assume the set ${\cal E}$ to have dimension strictly less than $d$. Moreover, let each trajectory $q(t)$, starting at $q_0 \in\,\reals^d\setminus {\cal E}$, be attracted to one of the stable equilibriums $O_i$, $i=1,\ldots,l$, as $t\to\infty$. Finally, assume that the projection of $b(q)$ on the radius connecting the origin in $\reals^d$ and the point $q \in\,\reals^d$ is directed to the origin and its length is bounded from below by some uniform constant $\theta>0$ (this condition provides the positive recurrence of the process $z_\e(t)=(q_\e(t),p_\e(t))$, $t\geq 0$).

In what follows, we shall denote
\[\begin{array}{l}
\ds{V(q_1,q_2)}\\
\vs
\ds{=\frac 12\inf\int_0^T\le|\a(f(s))\si^{-1}(f(s))\le(\dot{f}(s)-\frac{b(f(s))}{\a(f(s))}\r)\r|^2\,ds,\ f(0)=q_1,\ f(T)=q_2,\ T>0\big\}}
\end{array}\]
and
\[V_{ij}=V(O_i,O_j),\ \ \ \ i,j \in\,\{1,\ldots,l\}.\]
In a generic case, the behavior of the process $(q^\e(t),p^\e(t))$, on time intervals of order $\exp(\la\e^{-1})$, $\la>0$ and $0<\e<<1$, can be described by a hierarchy of cycles as in \cite{fw} and \cite{fr91}. The cycles are defined by the numbers $V_{ij}$. For (almost) each initial point  $q$ and a time scale $\la$, these numbers define also the metastable state $O_{i^\star}$, $i^\star=i^\star(q,\la)$, where $q^\e(t)$ spends most of the time during the time interval $[0,\exp(\la\e^{-1})]$. Slow changes of the field $b(q)$ and/or of the damping coefficient $\a(q)$ can lead to stochastic resonance (compare with \cite{fr00}).

\medskip

Consider next the reaction diffusion equation in $\reals^d$
\begin{equation}
\label{cf50}
\le\{\begin{array}{l}
\ds{\frac{\partial u}{\partial t}(t,q)={\cal L} u(t,q)+c(q,u(t,q))u(t,q),}\\
\vs
\ds{u(0,q)=g(q),\ \ \ \ q \in\,\reals^d,\ \ \ t>0.}
\end{array}\r.\end{equation}
Here ${\cal L}$ is a linear second order uniformly elliptic operator, with regular enough coefficients. Let $q(t)$ be the diffusion process in $\reals^d$ associated with the operator ${\cal L}$. The Feynman-Kac formula says that $u$ can be seen as the solution of the problem
\begin{equation}
\label{cf51}
u(t,q)=\E_q\,g(q(t))\exp\int_0^tc(q(s),u(t-s,q(s))\,ds.
\end{equation}

Reaction-diffusion equations describe the interaction between particle transport defined by $q(t)$ and reaction which consists of multiplication (if $c(q,u)>0$) and annihilation (if $c(q,u)<0$) of particles. 
In classical reaction-diffusion equations, the Langevin dynamics which describes a diffusion with inertia is replaced by its vanishing mass approximation. If the transport is described by the Langevin dynamics itself, equation \eqref{cf50} should be replaced by an equation in $\reals^{2d}$. Assuming that the drift is equal to zero ($b(q)=0$),  and the damping is of order $\e^{-1}$, as $\e\downarrow 0$, this equation has the form
\begin{equation}
\label{cf55}
\le\{\begin{array}{rl}
\ds{\frac{\partial u^\e}{\partial t}(t,q,p)=}  &  \ds{\frac 12\sum_{i,j=1}^d a_{i,j}(q)\frac{\partial^2 u^\e}{\partial p_i\partial p_j}(q,p)-\frac 1\e \a(q)p\cdot\nabla_p u^\e(q,p)+p\cdot \nabla_q u^\e(q,p) }\\
& \vs
& \ds{+c(q,u^\e(t,q,p))u^\e(t,q,p),\ \ \ \ t> 0,\ \ \ \ (q,p) \in\,\reals^{2d},}\\
&  \vs
\ds{u^\e(0,q,p)=}  & \ds{g(q)\geq 0,\ \ \ \ (q,p) \in\,\reals^{2d}.}
\end{array}\r.\end{equation}

Now, we define
\[R(t,q)=\sup\le\{\int_0^t c(f(s),0)\,ds-I_t(f)\, :\, f(0)=q,\ f(t) \in\,G_0\r\},\]
where 
\[I_t(f)=\frac 12\int_0^t\a^2(f(s))a^{-1}(f(s))\dot{f}(s)\cdot \dot{f}(s)\,ds,\]
and $G_0=\text{supp}\{g(q),\ q \in\,\reals^d\}$.

\begin{Definition}
\begin{enumerate}
\item We say that Condition (N) is satisfied if $R(t,x)$ can be characterized, for any $t>0$ and $x \in\,\Sigma_t=\{q \in\,\reals^d,\ R(t,q)=0\}$, as 
\[\begin{array}{l}
\ds{\sup\le\{\int_0^tc(f(s),0)\,ds-I_t(f),\ f(0)=q,\ f(t) \in\,G_0,\ R(t-s,f(s))\leq 0,\ 0\leq s\leq t\,\r\}.}
\end{array}\]
\item We say that the non-linear term $f(q,u)=c(q,u)u$ in equation \eqref{cf55} is of KPP (Kolmogorov-Petrovskii-Piskunov) type if $c(q,u)$ is Lipschitz-continuous, $c(q,0)\geq c(q,u)>0$, for any $0<u<1$, $c(q,1)=0$ and $c(q,u)<0$, for any $u>1$.
\end{enumerate}
\end{Definition}

\begin{Theorem}
Let the non-linear term in \eqref{cf55} be of KPP type. Assume that Condition (N) is satisfied and assume that the closure of $G_0=\text{{\em supp}}\{g(q),\ q \in\,\reals^d\}$ coincides with the closure of the interior of $G_0$. Then,
\begin{equation}
\label{cf61}
\lim_{\e\to 0} u^\e(t/\e,q,p)=0,\ \ \ \text{if}\ R(t,q)<0,
\end{equation}
and
\begin{equation}
\label{cf62}
\lim_{\e\to 0} u^\e(t/\e,q,p)=1,\ \ \ \text{if}\ R(t,q)>0,\end{equation}
so that equation $R(t,q)=0$ in $\reals^{2d}$ defines the interface separating the area where $u^\e$, the solution of \eqref{cf55}, is close to $1$ and to $0$, as $\e\downarrow 0$.
\end{Theorem}

\begin{proof}
If we define
$u_\e(t,q,p)=u^\e(t/\e,q,p)$, the analog of \eqref{cf51} yields
\begin{equation}
\label{cf60}
u_\e(t,q,p)=\E_{(q,p)} g(q_\e(t))\exp\le(\frac 1\e\int_0^tc(q_\e(s),u_\e(t-s,q_\e(s),p_\e(s))\,ds\r),\end{equation}
where
$z_\e(t)=(q_\e(s),p_\e(s))$ is the solution to equation \eqref{eq2}. By taking into account our assumptions on $c(q,u)$, we derive from \eqref{cf60}
\[u_\e(t,q,p)\leq \E_{(q,p)} \,g(q_\e(t))\exp\le(\frac 1\e\int_0^tc(q_\e(s),0)\,ds\r).\]
Theorem \ref{teo1} and the Laplace formula imply that the right hand side of the above inequality is logarithmically equivalent , as $\e\downarrow 0$, to $\exp\le(\frac 1\e R(t,q)\r)$ and this implies \eqref{cf61}.

In order to prove \eqref{cf62}, first of all one should check that if $R(t,q)=0$, then for each $\d>0$ 
\begin{equation}
\label{cf63}
u_\e(t,q,p)\geq \exp\le(-\frac 1\e\d\r),
\end{equation}
when $\e>0$ is small enough. This follows from \eqref{cf60} and Condition (N), if one takes into account the continuity of $c(q,u)$. The strong Markov property of the process $(q_\e(t),p_\e(t))$ and bound \eqref{cf63} imply \eqref{cf62} (compare with \cite{fr85}).

\end{proof}

Consider, as an example, the case $c(q,0)=c=\text{const}$. Then
\[R(t,q)=c t-\inf\le\{ I_t(f),\ f(0)=q,\ f(t) \in\,G_0\r\}.\]
The infimum in the equality above coincides with
\[\frac 1{2t}\rho^2(q, G_0),\]
(see, for instance, \cite{fr85} for  a proof), where $\rho(q_1,q_2)$, $q_1, q_2 \in\,\reals^d$, is the distance in the Riemaniann metric
\[ds=\a(q)\sqrt{\sum_{i,j=1}^d a_{i,j}(q)dq_i\,dq_j}.\]
This implies that the interface moves according to the Huygens principle with the constant speed $\sqrt{2c}$, if calculated in the Riemannian metric $ds$.

If $\a(q)=0$ in a domain $G_1\subset \reals^d$, the points of $G_1$ should be identified. The Riemaniann metric in $\reals^d$ induces now, in a natural way, a new metric $\tilde{\rho}$ in this space with identified points. The motion of the interface, in this case, can be described by the Huygens principle with constant velocity $\sqrt{2c}$ in the metric $\tilde{\rho}$. 

If $c(q,0)$ is not constant, the motion of the interface, in general, cannot be described by a Huygens principle. Actually, the motion can have jumps and other specific features (compare with \cite{fr85}). 

Finally, if the Condition (N) is not satisfied, the function $R(t,q)$ should be replaced by another one. Define
\[\tilde{R}(t,q)=\sup\le\{\min_{0\leq a\leq t}\le(\int_0^a c(f(s),0)\,ds-I_a(f)\r)\, :\, f(0)=q,\ f(t) \in\,G_0\r\}.\]
The function $\tilde{R}(t,q)$ is Lipschitz continuous and non-positive and if Condition (N) is satisfied, then 
\[\tilde{R}(t,q)=\min\le\{R(t,q),0\r\}.\]
By proceeding as in \cite{fr91}, it is possible to prove that
\[\lim_{\e\to 0} u^\e(t/\e,q,p)=0,\ \ \ \text{if}\ R(t,q)<0,\] and
\[\lim_{\e\to 0} u^\e(t/\e,q,p)=1,\]
if $(t,q)$ is in the interior of the set $\{(t,q)\,:\,t>0,\ q \in\,\reals^d,\ \tilde{R}(t,q)=0\}$.

\bigskip

Finally, we would like to mention a few generalizations.

\begin{enumerate}

\item The arguments that we we have used in  the proof of Theorem \ref{teo1}, can be used to prove the same result for the equation 
\[\le\{\begin{array}{l}
\ds{\ddot{q}^{\,\e}(t)=b(q^\e(t))-\frac {\a(q^\e(t))}\e\dot{q}^{\,\e}(t)+\frac 1{\e^\beta}\,\si(q^\e(t))\dot{B}(t),}\\
\vs
\ds{q^\e(0)=q \in\,\reals^{d},\ \ \ \ \dot{q}^\e(0)=p \in\,\reals^{d},}
\end{array}
\r.
\]
for any $\beta<1/2$. As a matter of fact, with the very same method we can show that also in this case the family $\{q_\e\}_{\e>0}$ satisfies a large deviation principle in $C([0,T];\reals^d)$ with action functional $I$ and with normalizing factor $\e^{1-2\beta}$.

\item The damping can be assumed to be anisotropic. This means that the coefficient $\a(q)$ can be replaced by a matrix $\a(q)$, with all eigenvalues having negative real part.

\item Systems with strong {\em non-linear damping} can be considered. Namely, let $(q^\e,p^\e)$ be the time-inhomogeneous Markov process corresponding to the following initial-boundary value problem for a degenerate quasi-linear equation on a bounded regular domain $G\subset \reals^d$
\[\le\{\begin{array}{ll}
\ds{\frac{\partial u^\e(t,p,q)}{\partial t}=}  &  \ds{\frac 12\sum_{i,j=1}^da_{i,j}(q)\frac{\partial ^2 u^\e(t,q,p)}{\partial p_i \partial p_j}+b(q)\cdot \nabla_pu^\e(t,q,p)}\\
&  \vs
&  \ds{-\frac {\a(q,u^\e(t,q,p))}\e p\cdot \nabla p u^\e(t,q,p)+p\cdot \nabla_q u^\e(t,q,p).}\\
&  \vs
\ds{u^\e(0,q,p)=g(q),}  &  \ds{\ \ \ \ \ \ u^\e(t,q,p)_{|_{q \in\,\partial G}}=\psi(q),}
\end{array}\r.\]
Existence and uniqueness of such degenerate problem, under some mild conditions, follows from \cite[Chapter 5]{fred}. The non-linearity of the damping leads to some pecularities in the exit problem and in metastability. In particular, in the generic case, metastable distributions can be distributions among several asymptotic  attractors and the limiting exit distributions may have a density (see \cite{kf}).

\end{enumerate}

\end{document}